\documentclass[12pt]{article}

\usepackage{amsthm}
\usepackage{amssymb}
\usepackage{amsmath}
\usepackage{amsfonts}
\usepackage{times}
\usepackage{color}
\usepackage{cite}
\usepackage{soul}
\usepackage{hyperref}
\hypersetup{
	colorlinks   = true,
	citecolor    = green,
	linkcolor    = blue
}

\usepackage[paper=a4paper, top=2.5cm, bottom=2.5cm, left=2.5cm, right=2.5cm]{geometry}

\def\qed{\hfill $\vcenter{\hrule height .3mm
		\hbox {\vrule width .3mm height 2.1mm \kern 2mm \vrule width .3mm
			height 2.1mm} \hrule height .3mm}$ \bigskip}

\def \RR {\mathbb R}
\def \NN {\mathbb N}
\def \EE {\mathbb E}

\def \CC {\mathbb C}
\def \PP {\mathbb P}
\def \eps {\varepsilon}
\def \vphi {\varphi}

\def \DC {\mathcal{C}_n}
\def \CC {\overline{\mathcal{C}}_n}

\def \TTT {\mathcal{T}}

\def \NN {\mathcal{N}}
\def \Stab {\mathrm{Stab}}
\def \DC {\mathcal{C}_n}
\def \INF {\mathrm{Inf}}
\def \MAXINF {\mathrm{MaxInf}}
\def \CC {\overline{\mathcal{C}}_n}

\def \CK {\mathcal{S}}

\newtheorem{theorem}{Theorem}
\newtheorem{lemma}[theorem]{Lemma}

\newtheorem{fact}[theorem]{Fact}
\newtheorem{claim}[theorem]{Claim}
\newtheorem{conjecture}[theorem]{Conjecture}
\newtheorem{proposition}[theorem]{Proposition}
\newtheorem{corollary}[theorem]{Corollary}
\theoremstyle{definition}

\theoremstyle{remark}

\long\def\symbolfootnotetext[#1]#2{\begingroup
	\def\thefootnote{\fnsymbol{footnote}}\footnotetext[#1]{#2}\endgroup}
\newcommand{\dm}[1]{{\color{red}{[[{\bf Dan:} #1]]}}}

\begin{document}
    \title{Noise stability on the Boolean hypercube via a renormalized Brownian motion}
	\author{Ronen Eldan\thanks{Microsoft Research and Weizmann Institute},~ Dan Mikulincer\thanks{MIT, supported by a Vannevar Bush Faculty Fellowship ONR-N00014-20-1-2826}~~and Prasad Raghavendra\thanks{UC Berkeley}}
    \maketitle
	\begin{abstract}
		We consider a variant of the classical notion of noise on the Boolean hypercube which gives rise to a new approach to inequalities regarding noise stability.  We use this approach to give a new proof of the Majority is Stablest theorem by Mossel, O'Donnell, and Oleszkiewicz, improving the dependence of the bound on the maximal influence of the function from logarithmic to polynomial. We also show that a variant of the conjecture by Courtade and Kumar regarding the most informative Boolean function, where the classical noise is replaced by our notion, holds true. Our approach is based on a stochastic construction that we call the renormalized Brownian motion, which facilitates the use of inequalities in Gaussian space in the analysis of Boolean functions.
	\end{abstract}
	\section{Introduction}
	Consider the discrete hypercube $\DC = \{-1,1\}^n$ equipped with the uniform measure $\mu^n$. Numerous inequalities concerning functions $f:\DC \to \RR$ involve the \emph{noise} operator, defined as follows: For each $x \in \DC$ and $\rho \in [0,1]$ let $\NN_{\rho,x}$ be the unique random variable taking values in $\DC$ whose coordinates are independent and such that $\EE[\NN_{\rho,x}] = \rho x$. We define the noise operator $T_\rho$ by the equation
	\begin{equation} \label{eq:defnoise}
		T_\rho[f](x) := \EE[f(\NN_{\rho,x})], ~~ \forall x \in \{-1,1\}^n.
	\end{equation}
	The quadratic form associated with this operator is sometimes called the \emph{noise stability} of $f$, denoted by
	$$
	\Stab_\rho(f) := \int_{\DC} f T_\rho[f] d \mu^n.
	$$
    We will mainly be interested in Boolean functions, $f: \DC \to \{0,1\}$. If $f$ is a Boolean function which satisfies $\EE[f] = \frac{1}{2}$, we will say that $f$ is balanced. For Booelan functions, noise stability is a canonical notion with far-reaching applications in social choice theory and theoretical computer science. 
    
	It is easy to check that the operator $T_\rho$ is diagonal with respect to the basis of characteristic functions of the form $\chi_A(x) := \prod_{i \in A} x_i$ and attains its spectral gap on functions of the form $x \to \mathbf{1}\{x_i > 0 \}$, called \emph{dictator} functions. This immediately implies the following fact.
	\begin{fact} \label{fact:stab}
		Among all balanced functions $f:\{-1,1\}^n \to [0,1]$, $\Stab_\rho(f)$ is maximized by the dictator function. 
	\end{fact}
	In this paper we are concerned with two important variants of the above fact. The first one is a deep theorem by Mossel, O'Donnell, and Oleszkiewicz \cite{MOO}, called the \emph{majority is stablest} theorem. It concerns the case that the function $f$ has small influences. Define
	$$
	\INF_i(f) := \int_{\DC} |f(x) - f(x^{\sim i})| d \mu^n(x), 
	$$
	as the $i$-th influence of $f$, where $x^{\sim i}$ stands for $x$ with the $i^{\mathrm{th}}$ bit flipped. Then, under the assumption
	$$
	\MAXINF(f) := \max_{i \in [n]} \INF_i(f) = o(1),
	$$
	we have that the \emph{majority function}, rather than the dictator, maximizes noise stability, up to a small error. 
	More precisely, define $\mathrm{Maj}_n:\DC \to \{0,1\}$, by $\mathrm{Maj}_n(x) = \mathbf{1} \bigl \{\sum\limits_{i=1}^n x_i \geq 0 \bigr \}.$ A straightforward application of the central limit theorem coupled with a computation in Gaussian space shows that,
	$$\lim\limits_{n\to \infty}\Stab_\rho(\mathrm{Maj}_n) = \frac{1}{4} + \frac{1}{2\pi} \arcsin(\rho).$$ It is also not hard to verify that $\MAXINF(\mathrm{Maj}_n) \xrightarrow{n \to \infty} 0$.
	With this in mind, the \emph{majority is stablest} theorem reads:
	\begin{theorem} (Majority is stablest, \cite{MOO}) \label{thm:moo}
		For every $\rho \in [0,1]$ and $\delta > 0$ there exists $\tau > 0$ such that for every balanced Boolean function $g:\DC \to \{0,1\}$ satisfying $\MAXINF(g) \leq \tau$, one has
		$$
		\Stab_\rho(g) \leq \frac{1}{4} + \frac{1}{2\pi} \arcsin(\rho) + \delta.
		$$
	\end{theorem}
	Here, the condition $\MAXINF(g) = o(1)$ can be thought of as ruling out functions that are not genuinely high-dimensional. Theorem \ref{thm:moo} was first conjectured in \cite{khot2007optimal}; its resolution in \cite{MOO} has been widely influential, with many implications for social choice theory, complexity, and Boolean analysis (see \cite[Section 2.3.2]{MOO} for some prominent examples).
	
	Another variant of Fact \ref{fact:stab}, which can be thought of as an \emph{entropic version},  with which we are concerned is a conjecture by Courtade and Kumar \cite{CK14}. Recall that the relative entropy between two variables $X,Y$ is defined as
	$$
	\mathrm{I}(X;Y) = H(X) - H(X | Y),
	$$
	where $H(X)$ is the Shannon entropy. The conjecture states,
	\begin{conjecture} \label{conj:CK} (Courtade-Kumar) 
		Let $X$ be uniformly distributed on $\DC$. Then, among all Boolean functions $f:\DC \to \{0,1\}$, the function which maximizes the quantity $\mathrm{I}(X; f(\NN_{X,\rho}))$ is the dictator function.
	\end{conjecture}
	Since its formulation, Conjecture \ref{conj:CK} has attracted the attention of many researchers from different disciplines. Several works, \cite{ordentlich2016improved,samorodnitsky2016entropy,yang2019most}, have managed to prove the conjecture in different noise regimes. However, a resolution of the conjecture, in full generality, remains elusive. The reader is referred to \cite{kindler2015remarks} for further details on the history of the conjecture and a discussion of the main difficulties.
	\subsection*{A new notion of noise}
	Our main contribution in this paper is the introduction of an alternative notion of noise, which is, on one hand, slightly different from the one given by the operator $T_\rho$, but on the other hand, can be used in order to derive bounds on the (usual) noise stability. We use this new notion of noise towards two main results:
	\begin{itemize}
		\item
		We give a new proof of the Majority is stablest theorem for the standard notion of noise.  Our proof yields an improved error bound in that it is polynomial in the maximal influence rather than logarithmic.
		\item
		We show that the variant of the Courtade-Kumar conjecture, which is obtained by replacing the classical notion of noise with ours, holds true. In fact, a more general statement holds where the Shannon (logarithmic) entropy can be replaced by an entropy with respect to any convex function.
	\end{itemize}

        Our notion of noise will be constructed as a martingale solution to a stochastic differential equation (SDE), driven by a Brownian motion. The main benefit of this approach is that it allows us to apply tools and techniques from stochastic analysis to analyze the process as a continuous function of the noise level. In particular, we perform a path-wise analysis of the process, which allows for finer manipulations than the analysis of the expectations in \eqref{eq:altnoise1}.

        Recently, similar techniques have proven to be quite useful in related settings. The paper \cite{eldan2018gaussian} used a Brownian motion, constrained to the cube, to study large deviations, and \cite{eldan2022concentration} as well as \cite{eldan2022log} constructed martingales to prove new concentration inequalities. In parallel to this work and relevant to the present setting, \cite{eldan2022optimal} employed a discrete process to prove a sharp version of the ``it ain't over till it's over'' theorem. Our technique is, in some sense, a continuation of these works, and we expand the use of stochastic analytic tools in the Boolean hypercube.
	\subsubsection*{Multilinear extensions and classical noise}
	To better understand our generalization of the notion of noise, let us first consider a slightly different point of view on the classical noise operator $T_\rho$. We first want to consider the \emph{multilinear extension} of a function $f:\DC \to \RR$. It is well-known that every such function $f$ can be uniquely written in the form
	$$
	f(x) = \sum_{A \subset [n]} \hat f(A) \chi_A,
	$$
	where $\chi_A(x) := \prod_{i \in A} x_i$, and the coefficients are obtained by the formula $\hat f(A) = \int f \chi_Ad\mu^n$. This form suggests that there is a natural way to extend the function $f$ from the discrete hypercube $\{-1,1\}^n$ to the continuous hypercube $[-1,1]^n$, simply by evaluating the above expression for $x_i$ taking values in $[-1,1]$. In what comes next, given a function $f:\DC \to \RR$, we will allow ourselves to also use the notation $f$ for the corresponding function defined on the continuous cube $\CC := [-1,1]^n$.
	
	Given this definition, it is now straightforward to check the identity
	$$
	T_\rho[f](x) = f\bigl (\rho x \bigr).
	$$
	In other words, we can replace the random point $\NN_{x, \rho}$ (which we think of as a noisy version of the point $x$) with the \emph{deterministic} point $\rho x$. 
	
	In light of this, we define $\mu^n_\rho$ to be the uniform measure on the set $\{-\sqrt{\rho}, \sqrt{\rho}\}^n = \sqrt{\rho}\DC$. Using this definition, we observe that
	\begin{equation}\label{eq:nsalt}
		\Stab_{\rho}(f) = \int_{\DC} f(x) f(\rho x) d \mu^n(x) = \int_{\sqrt{\rho}\DC} f(x)^2 d \mu^n_\rho(x).	
	\end{equation}
	Moreover, observe that
	\begin{equation} \label{eq:CKalt}
		\mathrm{I}(X; f(\NN_{x,\rho})) =  \int_{\DC} h(f(\rho x)) d \mu^n(x) - h\left(\int_{\DC} f(x) d\mu^n(x)\right),
	\end{equation}
	where $h(x) = x \log x + (1-x) \log(1-x)$ is the negative of the binary entropy function. This implies that Conjecture \ref{conj:CK} is equivalent to the fact that the quantity $\int h(f(x)) d \mu_\rho^n(x) - h\left(\int f(x) d\mu_\rho^n(x)\right)$ is maximized by dictator functions.
	
	\subsubsection*{Generalization of the notion of noise}
	This point of view, and in particular equation \eqref{eq:nsalt}, leads to the following generalization of the notion of noise: Given a symmetric probability measure $\nu$ on $[-1,1]$, we consider its $n$-fold tensor power, $\nu^n$, a measure on $\CC$, and define
	\begin{equation} \label{eq:altnoise1}
		\Stab_{\nu}(f) = \int_{\CC} f(x)^2 d \nu^{n}(x),
	\end{equation}
	so that, by \eqref{eq:nsalt}, we have $\Stab_\rho(f) = \Stab_{\mu^n_\rho}(f)$. 
	
	With this definition, we also get that
	\begin{align}
		\Stab_\nu(f) & = \int_{\CC} f(x)^2 d \nu^{n}(x) \nonumber \\
		& = \int_{\CC} \left (\sum_{A \subset [n]} \hat f(A) \chi_A(x) \right )^2 d \nu^{ n}(x) \nonumber \\
		& = \int_{\CC} \sum_{A,A' \subset [n]} \hat f(A) \hat f(A') \prod_{i \in A} \prod_{j \in A'} x_i x_j d \nu^{ n}(x) \nonumber\\
		& = \sum_{A \subset [n]} \hat f(A)^2 \left (\int x^2 d \nu (x) \right )^{|A|} = \Stab_{\mathrm{Var}[\nu]}(f). \label{eq:stabnu}
	\end{align}
	Due to this identity, inequalities regarding the classical noise stability can be alternatively proven for the quantity $\Stab_\nu(f)$ as long as $\nu$ is chosen with the correct variance.
	
	Next, we consider a more general notion of noise stability as follows: Let $\varphi:[0,1] \to \RR$ be convex and define
	\begin{equation} \label{eq:altnoise2}
		\CK_{\varphi,\nu} (f) = \int_{\CC} \varphi(f(x)) d \nu^{n}(x) - \varphi \left (\int_{\CC} f d \nu^n \right ).
	\end{equation}
	Using this definition, Fact \ref{fact:stab} is equivalent to the fact that $\CK_{x \to x^2, \mu^n_\rho}$ is maximized for the dictator function, and Conjecture \ref{conj:CK} is equivalent to the statement that $\CK_{h,\mu^n_\rho}$ is maximized for the dictator function, when $h$ is the binary entropy function.
	
	\subsubsection*{Our results}
	In this work, we will construct a family of measures $\nu_t$ for which extremality of the dictator and majority functions seems to arise naturally. We will first show the following variant of Conjecture \ref{conj:CK}:
	\begin{theorem} \label{thm:CKvariant}
		Let $\varphi:[0,1] \to \RR$ be convex and $t \in (0,\infty)$. Among all Boolean functions $f:\DC \to \{0,1\}$, the function which maximizes the quantity $\CK_{\varphi, \nu_t}$ is the dictator function.
	\end{theorem}
    Let us emphasize that Theorem \ref{thm:CKvariant} does not imply Conjecture \ref{conj:CK}, since, in line with the above discussion and as will become immediately apparent, $\nu_t \neq \mu^n_\rho$, for any $t \in (0,\infty)$ and $\rho \in (0,1)$. Moreover, the fact that Theorem \ref{thm:CKvariant} deals with a different noise model has one striking consequence, it holds for arbitrary convex functions. The analogous result, in the setting of Conjecture \ref{conj:CK}, is known to be false. For example (see \cite{ordentlich2016improved}), if $\vphi(x) = x^k$ for some $k \gg 2^n$, and $f$ is the dictator function, then 
    $$\CK_{x \to x^k, \mu_p^n}(\mathrm{Maj}_n) > \CK_{x \to x^k, \mu_p^n}(f).$$ 
    In this regard, the measures $\nu_t$ can be viewed as a natural family of measures for which dictators maximize all generalized notions of noise stability simultaneously.
    This property is reminiscent of results in Gaussian space, proven in \cite{kindler2015remarks}. In particular, \cite[Theorem 5.1]{kindler2015remarks} is another variant of Conjecture \ref{conj:CK}, with respect to standard Gaussian and the Ornstein–Uhlenbeck semigroup, and which holds for arbitrary convex functions.
	
    While $\nu_t$ is an alternative noise model, as discussed above, identity \eqref{eq:altnoise1} allows transferring noise stability bounds, with respect to $\nu_t$, to the classical noise model, which leads us to our next result. Our second goal in this work will be to give an upper bound for $\Stab_{\nu_t}(g)$ given an upper bound on $\MAXINF(g)$. This will imply a quantitative strengthening of Theorem \ref{thm:moo}. We prove our result for non-balanced functions. In this case, the quantity $\frac{1}{4} + \frac{1}{2\pi} \arcsin(\rho)$ in Theorem \ref{thm:moo} should be replaced by the Gaussian noise stability $\Lambda_\rho(\EE\left[g\right])$. This function is defined by,
	\begin{equation} \label{eq:gaussnoise}
		\Lambda_\rho(a):= \PP\left(G_1\leq \Phi^{-1}(a) \text{ and } G_2\leq \Phi^{-1}(a)\right),
	\end{equation}
	where $G_1,G_2$ are standard Gaussians on $\RR$ with $\EE\left[G_1G_2\right] = \rho$, and, for $a \in (0,1)$, $\Phi^{-1}(a)$ is the unique number satisfying $\PP\left(G_1 \geq a\right) = \Phi^{-1}(a)$.\\
    
     With the above definition, we show:
	\begin{theorem} \label{thm:majoritymain}
		Let $\rho \in [0,1]$ and let $g: \DC \to \{0,1\}$, with $\MAXINF(g) \leq \kappa$, for some $\kappa \geq 0$. Then,
		\begin{align*}
			\mathrm{Stab}_{\rho}(g) \leq \Lambda_\rho(\EE\left[g\right]) +\frac{C}{1-\sqrt{\rho}}\kappa^{\frac{1- \rho}{27}},
		\end{align*}
		for some numerical constant $C > 0$.
	\end{theorem}
 \iffalse
 To ease the notation we have stated Theorem \ref{thm:majoritymain} for Boolean functions taking values in $\{0,1\}$. Converting the result to $\{-1,1\}$-valued functions is straightforward via the transformation $x \to 2x-1$. In particular, a well-known calculation, \cite{sheppard1899application}, gives that $\Lambda_\rho(\frac{1}{2}) = \frac{1}{4} + \frac{1}{2\pi}\arcsin(\rho)$, for every $\rho \in (0,1)$.
 \fi
    The original proof of Theorem \ref{thm:moo}, in \cite{MOO}, gave the quantitative bound $$\delta = O\left(\frac{\log\log\frac{1}{\kappa}}{\log \frac{1}{\kappa}}\right).$$ Another proof, based on induction over the dimension, appears in \cite{de2016majority}, with a similar logarithmic dependency. Thus Theorem \ref{thm:majoritymain} should be seen as a quantitative improvement of Theorem \ref{thm:moo}, which affords polynomial bounds.
    There are also several generalizations of the majority is stablest theorem which appear in the literature, such as \cite{dinur2009conditional}, \cite{filmus2018invariance} and \cite{mossel2010gaussian}. It would be interesting to see if our technique is also applicable in those settings.
	\paragraph{Acknowledgements:} Work on this paper was initiated at the Simons institute's "Probability, Geometry, and Computation in High Dimensions" program. We thank the institute for its hospitality and the program's participants for contributing ideas to the polymath project from which this work eventually developed. We owe a special debt of gratitude to Joe Neeman, for many discussions and for taking an active part in the early stages of this work.
    %%%%%%%% Section 2: Stochastic Constructions
	\section{Stochastic constructions} \label{sec:constructions}
    
	At the center of this work is a stochastic process to sample bits.  
	To motivate the construction, let us begin with the standard Brownian motion.  Suppose we would like to sample a uniformly random bit $X \in \{-1,+1\}$ in a continuous incremental fashion.  One approach would be to run a Brownian motion starting at $0$ and output whichever of $\{-1,+1\}$ the Brownian motion reaches first.  Specifically, if we define a stochastic process,
	\[X(0) = 0, ~~~~  dX(t) = \mathbf{1}_{|X(t)| \leq 1} \cdot dB(t)\]
	for standard Brownian motion $B(t)$ on $\mathbb{R}$, stopping the walk when it reaches $\{-1,+1\}$.
	This stochastic construction has the property that the  expected variance of the output bit drops linearly with time.  Specifically, the variance of the output bit, namely $X(\infty)$, satisfies the following stochastic differential equation,
	\[ d \mathrm{Var}[X(\infty)|X(t)] = d\left[1-X(t)^2\right] = -2X(t) \cdot dB(t) - dt. \]
	In other words, $\mathrm{Var}[X(\infty)|X(t)]$ has a constant drift $dt$ along with a martingale term.  Intuitively, the process leaks variance at a constant rate.

    The stochastic process we consider is obtained by requiring that it leaks information (as measured by Shannon entropy) at a constant rate.  Formally, consider the following stochastic differential equation:
    \[X(0) = 0, ~~~~  dX(t) = \sqrt{(1+X(t))(1-X(t))} \cdot \mathbf{1}_{|X(t)| \leq 1} \cdot dB(t)\]
    If $\mathrm{H}(X(\infty)|X(t))$ denotes entropy, then it satisfies the stochastic differential equation with a constant drift, namely,
    	\[ d \mathrm{H}[X(\infty)|X(t)]] =  -\frac{1}{2}\log \left(\frac{(1+X(t))}{(1-X(t))}\right)  dB(t) - dt, \]
   and the information leaks at a constant rate.  It is not difficult to see that the above stochastic process is the unique martingale for which the entropy leaks at a constant rate.
    
    In higher dimensions, the natural generalization is to use an independent copy of the process for each coordinate.
    Formally, let $B(t)$ be a standard Brownian motion in $\RR^n$, and consider the martingale defined by the following stochastic differential equation:
	$$
	X(0) = 0, ~~~~ d X(t) = \sigma_t d B(t),
	$$
	where $\sigma_t$ is the diagonal matrix with 
	$$
	(\sigma_t)_{i,i} := \sqrt{(1+X_i(t)) (1-X_i(t))} \mathbf{1}_{|X_i(t)| \leq 1}.
	$$ 
        We call the process $X_t$ a \emph{renormalized Brownian motion} for reasons that will become more apparent later on. We note at this point that $(\sigma_t)_{i,i}^2 = \mathrm{Var}(X_i(\infty) | X_i(t))$ (see lemma \ref{lem:momentsprocess} below) which means that each coordinate is moving at a speed proportional to its remaining variance. As we will see later on, the behavior of $f(X(t))$ will be related to moments of the law of $X_i(\infty) | X_i(t)$, pushed forward by the map $\sigma_t^{-1}$. The choice of $\sigma_t$ ensures that this measure is \emph{isotropic}, which will play a key role in our proof.
        
        We define 
	$$
	\nu_t := \mathrm{law}(X_1(t)).
	$$
	This measure will correspond to our main notion of noise in view of equations \eqref{eq:altnoise1} and \eqref{eq:altnoise2}. A calculation gives the following,
	\begin{lemma} \label{lem:noisetonoise}
		One has, for all $t \geq 0$,
		\begin{equation} \label{eq:varnu}
			\mathrm{Var}[\nu_t] = 1-e^{-t}.
		\end{equation}
	\end{lemma}
 \begin{proof}
    By It\^o's formula, we have $d X_1(t)^2 = (\sigma_t)_{1,1} d B_1(t) + (1-X_1(t)^2) dt$, and therefore, $\frac{d}{dt} \EE X_1(t)^2 = \EE[1-X_1(t)^2]$, which implies that $\mathrm{Var}[X_t(t)] = 1-e^{-t}$.
 \end{proof}

	Fix a function $f:\DC \to [0,1]$ (which later on will be a Boolean function). At the center of our analysis is the process
	$$
	N_t := f(X(t)).
	$$
	This process is associated with noise stability via the relation
	$$
	\Stab_{\mathrm{Var}[\nu_t]}(f) \stackrel{\eqref{eq:stabnu}}{=} \EE[N_t^2],
	$$
	which, using \eqref{eq:varnu}, gives
	\begin{equation} \label{stabNt}
		\Stab_{\rho}(f) = \EE \left [N_{\log\frac{1}{1-\rho}}^2 \right].
	\end{equation}
	Moreover, we have, by definition,
	\begin{equation} \label{eq:CKNt}
		\CK_{\varphi,\nu_t} (f) = \EE[\varphi(N_t)] - \varphi \left (\int f d \mu^n \right ).
	\end{equation}
	These two identities will allow us to analyze noise stability and its generalization through the analysis of this process. An important property of the process is the following:
	\begin{fact}
		The process $N_t$ is a martingale.
	\end{fact}
 \begin{proof}
 Note that $N_t$ is a multilinear function of the martingale $X_1(t),...,X_n(t)$. Since those are independent martingales, every product of a subset of these processes is a martingale.
 \end{proof}
	In section \ref{sec:properties} we detail some further properties of the process $X(t)$, which shall prove useful in the proofs to come.\\
	
	Let us now introduce the main idea of our approach. In brief, our technique relies on comparing the evolution of the martingale $N_t$ to another martingale $M_t$, a so-called model process. The process $M_t$ will be constructed in a problem-dependent way, a process corresponding to the dictator function in the proof of Theorem \ref{thm:CKvariant} and a process in Gaussian space for Theorem \ref{thm:majoritymain}. At the heart of our proof is a coupling between the processes, used in \cite{eldan2015two}, whose existence is ensured by the following proposition.
	\begin{proposition} \label{prop:coupling}
		Let $(M_t)_{t \geq 0}$ and $(N_t)_{t \geq 0}$ be continuous martingales (defined on different probability spaces) such that $M_0 = N_0$. Then, these two processes can be defined over the same probability space, along with a process $(W_\tau)_{\tau \geq 0}$ such that $W_\tau$ is a standard Brownian motion with starting condition $W_0 = M_0$ and such that, almost surely,
		\begin{equation}
			M_t = W_{[M]_t}~~~ \mbox{and}~~~ N_t = W_{[N]_t}, ~~~ \forall t \geq 0.
		\end{equation}
	\end{proposition}
	\begin{proof}
		By invoking the Dambis / Dubins-Schwartz theorem, one may define a Brownian motion $(W_\tau)_{\tau \geq 0}$ over the same probability space as $M_t$, which satisfies $M_t = W_{[M]_t}$ for all $t \geq 0$. The same can be done for the process $N_t$ yielding a Brownian motion $\tilde W_t$. Finally, the four processes can be defined over one probability space in a way that $W_\tau = \tilde W_\tau$ almost surely by invoking \cite[Theorem 10]{eldan2015two}.
	\end{proof}
	\section{A variant of the Courtade-Kumar conjecture}
	In this section we prove Theorem \ref{thm:CKvariant}. We begin by defining a model process $M_t$ and establishing a comparison principle. 
	
	\paragraph{Comparison with dictators:}
	Let $f(x) = \mathbf{1}\{x_i > 0 \}$ be the dictator function and let $g:\DC \to \{0,1\}$ be a balanced Boolean function. That is, $\int g d\mu^n = \frac{1}{2}$. Consider the martingales $N_t = f(X(t)), M_t = g(\tilde X(t))$, where $\tilde X(t)$ denotes a process having the same distribution as that of $X(t)$ but which lives on a different probability space.
	
	The proof of Theorem \ref{thm:CKvariant} will proceed by comparing the evolution of the processes $N_t$ and $M_t$. The next claim will be our main vehicle for the comparison. 
	\begin{claim} \label{claim:1}
		We have,
		$$
		\frac{d}{dt} [M]_t = \|\sigma_t \nabla f(X(t))\|_2^2 = f(X(t))(1-f(X(t))) = M_t(1-M_t)
		$$
		and
		$$
		\frac{d}{dt} [N]_t = \|\sigma_t \nabla g(\tilde X(t))\|_2^2 \leq g(\tilde X(t)) (1-g( \tilde X(t))) = N_t(1-N_t).
		$$
	\end{claim}
	\begin{proof}
		Ito's formula gives,
		$$
		d M(t) = \nabla f(X(t)) d X(t) = \nabla f(X(t)) \sigma_t d B_t,
		$$
		which implies that
		$$
		d[M]_t = \|\sigma_t \nabla f(X(t))\|_2^2 dt.
		$$
		Note that we have the following multilinear representation for the dictator, $f(x) = \frac{x_1 + 1}{2}$. In particular, $\nabla f = (\frac{1}{2},0,\dots,)$, and
		$$\frac{d}{dt}[M]_t = \|\sigma_t \nabla f(X(t))\|_2^2 = \frac{1}{4} |(\sigma_t)_{1,1}|^2 = \frac{1}{4}(1-X_1(t)(1+X_1(t)) =  f(X(t))(1-f(X(t))).$$
		Similarly, for $N_t$, 
		$$
		\frac{d}{dt}[N]_t = \|\sigma_t \nabla g(\tilde{X}(t))\|_2^2 = \sum\limits_{i \in [n]} (1- \tilde{X}(t)^2)\partial_ig(\tilde{X}(t))^2.
		$$
		The inequality is a consequence of the following version of Parseval's inequality, which we shall prove below:
     \begin{equation} \label{eq:biasedpars}
     g(x) \geq g(x)^2 + \sum\limits_{i=1}^n(1-x_i^2)(\partial_ig(x))^2.
    \end{equation}
        Rearranging the terms in \eqref{eq:biasedpars} gives:
		$$
		g(x)(1-g(x)) \geq \sum_{i \in [n]} (1-x_i^2)  (\partial_i g(x))^2, %\geq \sum_{i \in [n]} 1_{|x_i| < 1 } (1+x_i)(1-x_i)  (\partial_i g(x))^2.
		$$
    which is the desired claim when $x = \tilde{X}(t)$.
    To prove \eqref{eq:biasedpars} note that, as $g$ is multilinear extension and since $X_t$ is a martingale, 
    $$g(x) = \EE\left[g(X(\infty))|X(0) = x\right],$$
    and $Y:=X(\infty)|X(0) = x$ is supported on $\{-1,1\}$ with $\PP\left(Y = 1\right) = \frac{1 + x}{2}$, as in Lemma \ref{lem:momentsprocess}. 
    For $A \subset [n]$, define the biased characters,
    $$\tilde{\chi}_A(y) = \prod_{i \in A}\frac{y - x}{1-x^2}.$$
    The set $\{\tilde{\chi}_A\}_{A \subset [n]}$ is orthonormal with respect to the law of $Y$ and induces the Fourier decomposition,
    $$g = \sum\limits_{A \subset [n]}  \hat{\tilde{g}}(A)\tilde{\chi}_A,$$
    with $\hat{\tilde{g}}(A) := \EE\left[g(Y)\tilde{\chi}_A(Y) \right]$. Thus, from orthonormality we deduce the following identity,
    $$g(x)^2 = \EE\left[g(Y)\right]^2 = \EE\left[g(Y)\tilde{\chi}_{\emptyset} (Y)|\right]^2=\hat{\tilde{g}}(\emptyset)^2.$$
    Similarly, by rewriting $\partial_ig$ in terms of the Fourier coefficients, e.g. as in \cite[Equation (8.7)]{odonnell2014analysis}, 
    $$(1-x_i^2)(\partial_ig(x))^2  = (1-x_i)^2\EE\left[\partial_ig(Y)\right]^2 = \EE\left[g(Y)\tilde{\chi}_{\{i\}} (Y)\right]=\hat{{\tilde{g}}}(\{i\})^2.$$
    Combining the last two identities, noting that $\EE\left[g(Y)^2\right] = \EE\left[g(Y)\right]=g(x)$, and applying Parsevel's inequality, we obtain \eqref{eq:biasedpars}.
	\end{proof}
	We are finally ready to prove the variant of the Courtade-Kumar conjecture.
	\begin{proof} [Proof of Theorem \ref{thm:CKvariant}]
		Apply Proposition \ref{prop:coupling} to the processes $M_t$ and $N_t$ defined in the previous section, so that
		\begin{equation}\label{eq:dds}
			M_t = W_{[M]_t}, ~~ N_t = W_{[N]_t}, ~~ \forall t \geq 0.
		\end{equation}

		Let $\mathcal{T}_1(\cdot), \mathcal{T}_2(\cdot)$ be the inverse of the non-decreasing functions $t \to [M]_t$ and $t \to [N]_t$ respectively. Consider the stopping time,
		$$
		\tau_{\max} := \min \{t \geq 0; |W_t| = 1 \}. 
		$$ 
		Then for all $0 \leq \tau \leq \tau_{\max}$, we have by \eqref{eq:dds} that
		$$
		M_{\mathcal{T}_1(\tau)} = N_{\mathcal{T}_2(\tau)} = W_\tau.
		$$
		An application of Claim \ref{claim:1} gives
		$$
		\mathcal{T}_1'(\tau) = \frac{1}{ M_{\mathcal{T}_1(\tau)}(1-M_{\mathcal{T}_1(\tau)}) } = \frac{1}{W_\tau(1-W_\tau)} \leq \mathcal{T}_2'(\tau), ~~ \forall 0 \leq \tau \leq \tau_{\max}.
		$$
		By integrating this inequality, we have $\mathcal{T}_1(\tau) \leq \mathcal{T}_2(\tau)$, almost surely, for all $0 \leq \tau \leq \tau_{\max}$. \\
		
		Now for a fixed $t_0 \geq 0$ set $T_1 = [M]_{t_0}, T_2 = [N]_{t_0}$. By the above, we have almost surely that $T_1 \geq T_2$. If $\varphi$ is convex, by Jensen's inequality we have,
		\begin{align*}
			\EE [\varphi(f(X_{t_0}))] &= \EE[\varphi(W_{T_1})] = \EE [\varphi(W_{T_1}) | W_{T_2}] \\
			&\geq \EE [\varphi(W_{T_2})] = \EE [ \varphi(g(X_{t_0}) ].
		\end{align*}
		This completes the proof.
	\end{proof}
	\section{Majority is stablest}
	In this section we prove Theorem \ref{thm:majoritymain}. The proof relies on the same kind of coupling as the one used in the previous section, but the process $M_t$ will be different - it will correspond to the majority function rather than to dictator. In fact, in light of the central limit theorem, we may define $M_t$ in a way that mimics a corresponding one-dimensional version of the process on Gaussian space. 
    In what comes next, we work with the processes $X(t)$ and $N_t$, defined as in section \ref{sec:constructions}.
	
	\subsection*{The model process}
	Denote by $\Phi:\RR\to \RR$  the standard Gaussian cumulative distribution function (CDF),
	$$\Phi(x) = \frac{1}{\sqrt{2\pi}}\int\limits_{x}^\infty e^{-t^2/2}dt,$$
	with $\Phi^{-1}$ its inverse and $\Phi'$ the standard Gaussian density.
	Define
	$$
	I(s) = \Phi'(\Phi^{-1}(s)),
	$$
	the Gaussian isoperimetric profile. Define the process $N_t$ to be the unique martingale satisfying,
	\begin{equation} \label{eq:QVM}
		M_0 = N_0, ~~ \frac{d}{dt} [M]_t = I(M_t)^2.
	\end{equation}
	This process was introduced in \cite{eldan2015two}, to study Gaussian noise stability, as defined in \eqref{eq:gaussnoise}. We use the following identity, which connects between the quadratic variation of $N_t$ and Gaussian noise stability (see \cite[Section 2.3]{eldan2015two} for the derivation).
	\begin{lemma} \label{lem:isopermitrytonoise}
		Let $t \geq 0$, then
		$$\EE\left[M_t^2\right] = \Lambda_{1-e^{-t}}(M_0),$$
		where $\Lambda$ stands for Gaussian noise stability, defined in \eqref{eq:gaussnoise}.
	\end{lemma}
	The representation in Lemma \ref{lem:isopermitrytonoise} allows to deduce Lipschitz bounds on $\Lambda$.
	\begin{lemma} \label{lem:lipschitzlambda}
		Let $0\leq \rho <\rho' \leq 1$. Then, for every $a \in (0,1)$,
		$$|\Lambda_{\rho'}(a) - \Lambda_{\rho}(a)| \leq \frac{\rho' - \rho}{1-\rho'}.$$
	\end{lemma}
	\begin{proof}
		Set $t' = \ln\left(\frac{1}{1-\rho'}\right)$ and $t = \ln\left(\frac{1}{1-\rho}\right)$. Suppose that $M_0 = a$. By Lemma \ref{lem:isopermitrytonoise},
		\begin{align*}
			|\Lambda_{\rho'}(a) - \Lambda_{\rho}(a)| = \EE\left[[M]_{t'}\right] - \EE\left[[M]_t\right] = \EE\left[\int\limits_t^{t'}I(M_s)^2ds\right] \leq t'-t = \ln\left(\frac{1}{1-\rho'}\right) - \ln\left(\frac{1}{1-\rho}\right),
		\end{align*}
		where the inequality uses $I(x) \leq 1$. To complete the proof, note that $\ln(\frac{1}{1-x})' = \frac{1}{(1-x)}$, which implies the desired bound through integration.
	\end{proof}
	
	\subsection*{Level-$1$ inequality}
	The next ingredient in the proof will be an analogue of Claim \ref{claim:1}. In other words, we would like to show that, in a sense, the process $N_t$ moves faster than the process $M_t$, thus accumulating more quadratic variation. However, in this case, the situation is substantially more involved, as we know that the process assigned with the dictator function actually moves faster than the process associated with the majority function (this is essentially the content of Claim \ref{claim:1}). 
	
	It turns out that as long as the condition $\|\nabla f(X(t))\|_\infty = o(1)$ holds, the majority function has the largest quadratic variation. This is the content of the following theorem.
	\begin{theorem}[Level-1 inequality] \label{thm:level1}
		Fix $t>0$ and define
		\begin{equation} \label{eq:defkappat}
			\kappa_t := \max_i (\partial_i g(X(t)))^2.    
		\end{equation}
		Then,
		$$
		|\sigma_t \nabla g(X(t))| \leq C_1\frac{\sqrt{\kappa_t}}{I(g(X(t)))} + I(g(X(t))),
		$$
		for some absolute constant $C_1 > 0$.
	\end{theorem}
	Using It\^o's formula, we have
	$$
	d [N]_t = |\sigma_t \nabla g(X(t))|^2 dt, 
	$$
	which immediately implies the bound
	\begin{equation} \label{eq:boundlvl1}
		d [N]_t \leq \left ( C_1\frac{\sqrt{\kappa_t}}{I(N_t)} + I(N_t) \right )^2 dt.
	\end{equation}
	
	In order to be able to apply this theorem, we will have to somehow argue that the process $\kappa_t$ (defined in its formulation) is small enough. For this, we will need a variant of the hypercontractivity inequality that holds for our notion of noise. We will use it to prove the following bound, for a noisy version of $g$.
	
	\begin{lemma} \label{lem:smallinfluences}
		Let $g : \{-1,1\}^n \to \{0,1\}$ be such that $\MAXINF(g)\leq \kappa$. Fix $\eps > 0$ and set $g_\eps = T_{e^{-\eps}}g$. Then, for all $\alpha, t > 0$, one has
		$$
		\PP \left (\max_{s \in [0,t], i \in [n]} |\partial_i g_\eps(X_s)| \geq \alpha \right ) \leq \left(\frac{1}{\alpha}\right)^{2 + e^{-t}}\frac{\kappa^{\frac{e^{-t}}{2}}}{\eps^2}\mathrm{Var}(g).
		$$
	\end{lemma}
	The two above bounds will be proven in Sections \ref{sec:level1} and \ref{sec:influences}, respectively.
	
	\subsection{Proof of Theorem \ref{thm:majoritymain}}
	We now have all the ingredients we need in order to obtain the improvement of the majority is stablest theorem. We first fix $\eps > 0$ to be some small number, and consider the noisy function $g_\eps:=T_{e^{-\eps}}g$. Note that $\EE_{\mu^n}[g_\eps] = \EE_{\mu^n}\left[g\right]$ and that $\mathrm{MaxInf}(g_\eps) \leq \mathrm{MaxInf}(g)$. Thus, with a slight abuse of notation, we consider the process with respect to the noisy function $N_t:= g_\eps(X(t)).$ We shall deduce noisy stability bounds for $g_\eps$, which, in turn, will imply a similar inequality for $g$.
	
	We apply Proposition \ref{prop:coupling} in order to define the processes $M_t$ and $N_t$ on the same probability space, along with a Brownian motion $(W_\tau)_{\tau \geq 0}$, such that
	$$
	M_t = W_{[M]_t}~~~ \mbox{and}~~~ N_t = W_{[N]_t}, ~~~ \forall t \geq 0.
	$$
	Equations \eqref{eq:QVM} and \eqref{eq:boundlvl1} give
	$$
	\frac{d}{dt} [M]_t = I(W_{[M]_t})^2, ~~ \frac{d}{dt} [N]_t \leq \left ( C_1\frac{\sqrt{\kappa_t}}{I(W_{[N_t]})} + I(W_{[N_t]}) \right )^2,
	$$
	where $\kappa_t$ is defined as in \eqref{eq:defkappat}. Our objective is to integrate those two inequalities while applying Lemma \ref{lem:smallinfluences} to ensure that $\kappa_t$ is small.
	
	Let $\beta> 0$ be a constant whose value will be chosen later. Consider the stopping times
	$$
	\tau = \inf \left \{t\geq 0: \kappa_t \geq \kappa^{\frac{2}{\beta}} \right\},
	$$
	and
	$$
	\tau' = \inf\{t\geq 0: I(N_t) \leq \kappa^{\frac{1}{4\beta}}\}.
	$$
	Let $\TTT(t)$ be the inverse of the monotone increasing function $t \to [M]_t$. Then, by the chain rule, we have for all $t \leq \tau' \wedge \tau$,
	\begin{align*}
		\frac{d}{dt} \TTT([N]_t) & \leq \frac{\left ( C_1\frac{\sqrt{\kappa_t}}{I(N_t)} + I(N_t) \right )^2}{I(N_t)^2} \\
		& \leq 1 + C_1^2 \frac{\kappa_t}{I(N_t)^4} + 2C_1\frac{\sqrt{\kappa_t}}{I(N_t)^2}\\
		& \leq 1 + 3 C^2_1 \kappa^{\frac{1}{2\beta}},
	\end{align*}
	where we used the fact that $I(\cdot) \leq 1$ and $\kappa_t \leq 1$ almost surely. It follows that, whenever $t \leq  \tau \wedge \tau'$,
	\begin{align*}
		[M]_t & = [M]_{\TTT([N]_t)} + [M]_t - [M]_{\TTT([N]_t)} \\
		& \geq [N]_t - \left ( [M]_t - [M]_{\TTT([N]_t)} \right )_+ \\
		& = [N]_t - \left ( \int_t^{\TTT([N]_t)} I(M_s)^2 ds \right )_+ \\
		& \geq [N]_t - \left ( \TTT([N]_t) - t \right )_+ \\
		& \geq [N]_t - \int_0^t \left (\frac{d}{ds} \TTT([N]_s) - 1 \right )_+ ds \geq [N]_t - 3 C^2_1 t \kappa^{\frac{1}{2\beta}}.
	\end{align*}
	Moreover, note that
	\begin{align*}
		\EE \left | [N]_{t} - [N]_{t \wedge \tau'} \right | & \leq \EE \left[ [N]_{\infty} - [N]_{\tau'} \right] \\
		& = \EE \bigl[ N_{\tau'} ( 1 - N_{\tau'} ) \bigr]\leq \EE\left[I(N_{\tau'})\right] \leq \kappa^{\frac{1}{4\beta}}.    
	\end{align*}
	where we use the fact that $I(s) \geq s(1-s)$. Combining the two last displays gives
	\begin{align*}
		\EE [N]_t & \leq \EE \left[[N]_t \mathbf{1}_{ \{t < \tau \wedge \tau' \}} \right] + \EE \left[ [N]_{t} \mathbf{1}_{ \{t \geq \tau \}} \right] + \EE \left[ [N]_{t \wedge \tau'} \mathbf{1}_{ \{\tau' \leq t \leq \tau \}} \right] +\EE \left | [N]_{t} - [N]_{t \wedge \tau'} \right | \\
        & \leq \EE \left[[N]_{t \wedge \tau \wedge \tau'} \right] + \EE \left[ [N]_{t} \mathbf{1}_{ \{t \geq \tau \}} \right] +\EE \left | [N]_{t} - [N]_{t \wedge \tau'} \right | \\  
		& \leq  \EE [M]_t + 3 C^2_1 t \kappa^{\frac{1}{2\beta}} +t\PP(t > \tau)  + \kappa^{\frac{1}{4\beta}}\\
		&\leq \EE [M]_t + t\PP(t > \tau) + (3C^2_1t + 1)\kappa^{\frac{1}{4\beta}}.
	\end{align*}
	To bound $\PP(t > \tau)$, we apply Lemma \ref{lem:smallinfluences} with $\alpha = \kappa^{\frac{1}{\beta}},$ and choose $\beta = \frac{4e^{-t} +5}{4e^{-t}}$, so that,
	$$\PP(t > \tau) \leq \left(\frac{1}{\alpha}\right)^{1 + \frac{e^{-t}}{2}}\frac{\kappa^{\frac{e^{-t}}{2}}}{\eps^2}\mathrm{Var}(g_\eps) \leq \frac{\kappa^{\frac{e^{-t}}{2}(1 - \frac{1}{\beta})-\frac{1}{\beta}}}{\eps^2} = \frac{\kappa^{\frac{1}{4\beta}}}{\eps^2}.$$
	Coupled with the previous computation, we obtain,
	$$\EE [N]_t \leq \EE [M]_t + \left(\frac{3C_1^2}{\eps^2}t + 1\right)\kappa^{\frac{1}{4\beta}}.$$
	Now, by Definition of $g_\eps$, and since $T_\cdot$ is a self-adjoint semigroup, we have,
	\begin{align*}
		\mathrm{Stab}_{\rho}(g_\eps) = \mathrm{Stab}_{\rho}(T_{e^{-\eps}}g) =  \int_{\DC} T_{e^{-\eps}}g T_\rho[T_{e^{-\eps}}g] d \mu^n = \int_{\DC} g T_{e^{-2\eps}\rho}[g] d \mu^n = \mathrm{Stab}_{e^{-2\eps}\rho}(g).
	\end{align*}
	Thus, since the above is true for any $\rho$ and $\eps$, from the relation \eqref{stabNt}, we see
	\begin{align*}
		\mathrm{Stab}_{e^{-2\eps}\rho}(g) = \EE \left[N_{\log\left(\frac{1}{1-\rho}\right)}^2 \right] \implies 	\mathrm{Stab}_{\rho}(g) = \EE \left[N_{\log\left(\frac{1}{1-e^{2\eps}\rho}\right)}^2 \right] ,
	\end{align*}
provided that $e^{2\eps}\rho<1$. This observation, along with Lemma \ref{lem:isopermitrytonoise}, then gives,
	\begin{align*}
		\mathrm{Stab}_{\rho}(g) &\leq \EE\left[M_{\log\left(\frac{1}{1-e^{2\eps}\rho}\right)}^2 \right] + \left(\frac{3C_1^2}{\eps^2}\log\left(\frac{1}{1-e^{2\eps}\rho}\right) + 1\right)\kappa^{\frac{1}{4\beta}}\\
		&= \Lambda_{e^\eps\rho}(\EE\left[g\right]) + \left(\frac{3C_1^2}{\eps^2}\log\left(\frac{1}{1-e^{2\eps}\rho}\right) + 1\right)\kappa^{\frac{1}{4\beta}}\\
		&\leq  \Lambda_{\rho}(\EE\left[g\right]) + \rho\frac{e^\eps - 1}{1-e^\eps\rho}+ \left(\frac{3C_1^2}{\eps^2}\log\left(\frac{1}{1-e^{2\eps}\rho}\right) + 1\right)\kappa^{\frac{1}{4\beta}}.
	\end{align*}
	where the second inequality is the Lipschitz bound from Lemma \ref{lem:lipschitzlambda}. Choose now $\eps = \kappa^{\frac{1}{12\beta}}$ and assume $\eps \leq \frac{1}{8}\ln(\frac{1}{\rho})$ to get
	\begin{align*}
		\mathrm{Stab}_{\rho}(g) \leq \Lambda_{\rho}(\EE\left[g\right]) +\frac{C}{1-\sqrt{\rho}}\kappa^{\frac{1}{12\beta}},
	\end{align*}
for some universal constant $C > 0$. Finally, since $t = \log\left(\frac{1}{1-\rho}\right)$, and since we have made the choice $\beta = \frac{4e^{-t} +5}{4e^{-t}} = \frac{9 - 4\rho}{4 - 4\rho}$, we have, 
$$\kappa^{\frac{1}{12\beta}} \leq  \kappa^{\frac{1-\rho}{27}},$$
which finishes the proof.

    \section{Level-$1$ inequalities} \label{sec:level1}

\def \vol {\mathrm{vol}}

Our goal in this section is to prove Theorem \ref{thm:level1}.  To this end, we will first prove a level-$1$ inequality for isotropic product measures.

\subsection{Level-$1$ inequality for subsets of isotropic product measures}

\begin{theorem} \label{thm:isotropiclevelone}
Let $\tilde{\xi}$ be an isotropic product distribution over $\mathbb{R}^n$.
So if $(Y_1,\ldots,Y_n)$ is drawn from $\tilde{\xi}$, $\EE[Y_i] = 0$ and $\EE[Y_i^2] = 1$.
Fix a function $S : \mathbb{R}^n \to [0,1]$ and let $\vol(S) = \EE_{Y \sim \tilde{\xi}}[S(Y)]$.
Then we have,
\begin{enumerate}
\item For every unit vector $\theta \in \mathbb{R}^n$, $\|\theta\|_2 = 1$,
\[ \int \langle \theta, y \rangle S(y) d\tilde{\xi}(y) \leq   I(\vol(S)) + C_1  \max_i (|\theta_i| \cdot \EE[Y_i^4]^{1/2}) \]
where $I : [0,1] \to [0,1]$ is defined as $I(s) = \Phi'(\Phi^{-1}(s))$
        for the Gaussian cumulative distribution function (CDF) $\Phi$ and $C_1$ is an absolute constant.
\item Moreover, 
\[
    \left\| \int y \cdot S(y) d \tilde \xi(y)\right \|_2 \leq I(\vol(S)) + \frac{C_1 \cdot \max_i(|\EE[Y_i \cdot S(Y)]| \cdot \EE[Y_i^4]^{1/2})}{I(\vol(S))} \]

\end{enumerate}
\end{theorem}

\noindent	Before delving into the proof of Theorem \ref{thm:isotropiclevelone}, let us first begin with a technical Lemma.
	\begin{lemma} \label{lem:symmetr}
		Let $\eta$ be a measure on $\RR^n$ and let $m:\RR^n \to [0,1]$. 
		Fix $\theta \in \RR^n$ and let $\alpha \in \RR$ be such that
		$$\int_{\RR^n}m(x)d\eta(x) \leq \eta\left(\left\{x|\langle x,\theta \rangle \geq \alpha \right\}\right) ,$$
		when $\alpha \geq 0$, and,
		$$\int_{\RR^n}m(x)d\eta(x) \geq \eta\left(\left\{x|\langle x,\theta \rangle \geq \alpha \right\}\right) ,$$
		when $\alpha < 0$.
		Then,
		$$\int_{\RR^n}\langle x, \theta\rangle m(x)d\eta(x) \leq \int_{\RR^n}\langle x, \theta\rangle {\bf 1}_{\langle x, \theta \rangle \geq \alpha}d\eta(x).$$
	\end{lemma}
	\begin{proof}
		Observe that with the above choice of $\alpha$,
		$$\alpha\int_{\RR^n}m(x)d\eta(x) \leq \alpha\int_{\RR^n} {\bf 1}_{\langle x, \theta \rangle \geq \alpha}d\eta(x).$$
		Hence,
		$$\int_{\RR^n}\langle x, \theta\rangle (m(x)-{\bf 1}_{\langle x, \theta \rangle \geq \alpha})d\eta(x) \leq \int_{\RR^n}(\langle x, \theta\rangle -\alpha) (m(x)-{\bf 1}_{\langle x, \theta \rangle \geq \alpha})d\eta(x).$$
		Since $0\leq m(x) \leq 1$, we have that  $(m(x) - {\bf 1}_{\langle x, \theta \rangle \geq \alpha}) \leq 0$ if and only if $(\langle x, \theta\rangle -\alpha) \geq 0$. Thus,
		$$\int_{\RR^n}(\langle x, \theta\rangle -\alpha) (m(x)-{\bf 1}_{\langle x, \theta \rangle \geq \alpha})d\eta(x)\leq 0,$$
		which completes the proof.
	\end{proof}

We are now ready to show our level-1 inequality for isotropic product measures.
\begin{proof}[Proof of Theorem \ref{thm:isotropiclevelone} (1)]
    By Lemma \ref{lem:symmetr}, we have,
		$$
		 \int \langle \theta, y \rangle S(y) d\tilde \xi(y)  \leq \int_\alpha^\infty t d \nu(t)
		$$
		where $\nu$ denotes the push-forward of $\tilde \xi_t$ under the map $x \to \langle x, \theta \rangle$ and $\alpha$ satisfies: 
		\begin{equation} \label{eq:alphaval}
			\nu\left(\left\{x| x\geq \alpha\right\}\right) \geq \vol(S) ,
		\end{equation}
		when $\alpha \geq 0$, and the reverse inequality when $\alpha < 0$. We continue the proof assuming $\alpha \geq 0$. The proof when $\alpha < 0$ is analogous, and there is no loss of generality.
	
	%	Therefore, if $\gamma$ stands for the standard Gaussian law, we have
	%	\begin{align} \label{eq:level1sep}
	%		| \sigma_t \nabla g(X(t)) | &\leq \int_\alpha^\infty t d \nu(t).
	%	\end{align}
    
    We denote the CDF of $\nu$ by $F_\nu(t) :=  \nu\left(\left\{x| x\geq t\right\}\right)$ and its quantile function by $Q_\nu(t):= \sup\left\{x\in \RR| t\geq F_\nu(x)\right\}$.
		Now, let $T: = Q_\nu\circ \Phi$.
		%$$ T(x) = \sup\{y \in \RR|\gamma\left([x,\infty)\right)) \leq \nu\left([y,\infty)\right)), \text{ for every } x\in \RR.$$
		It is readily verified that $T$ is non-decreasing and that $T(z)$ has the law $\nu$ if $z$ is a Gaussian. 
		
			Moreover, if $x \in \RR$ is such that $T(x) \geq \alpha$, then $x \geq \Phi^{-1}(\vol(S))$. Indeed, by the definition of $Q_\nu$, we have,
		$$Q_\nu(\Phi(x)) \geq \alpha \implies \Phi(x) \geq F_\nu(\alpha) \implies \Phi(x) \geq \vol(S) \implies x \geq \Phi^{-1}(\vol(S)),$$
		where the second implication is \eqref{eq:alphaval}.
		So, with a change of variables,
		\begin{align} \label{eq:level1T}
		\int_\alpha^\infty t d \nu(t)
& =  \int\limits_{T^{-1}([\alpha, \infty))} T(t)d\gamma(t) \\
& \leq \int\limits_{\Phi^{-1}(\vol(S))}^\infty T(t)d\gamma(t)\nonumber\\
			&\leq \int\limits_{\Phi^{-1}(\vol(S))}^\infty td\gamma(t) + \int\limits_{-\infty}^\infty |t-T(t)|d\gamma(t)\nonumber\\
			&=  I(\vol(S)) + \int\limits_{-\infty}^\infty |t-T(t)|d\gamma(t),
		\end{align}
		where the equality uses integration by parts.
		Thus, to finish the proof, we need to bound 
		$$\int\limits_{-\infty}^\infty |t-T(t)|d\gamma(t)\leq \sqrt{\int\limits_{-\infty}^\infty |t-T(t)|^2d\gamma(t)}.$$
		Observe that since $T$ is monotone, it is a derivative of a convex function. Hence, by \cite[Theorem 9.4]{villani2008optimal}, $T$ is the optimal-transport map from $\gamma$ to $\nu$. In other words,
		\begin{equation} \label{eq:wasdef}
			\mathbf{W}_2(\gamma,\nu) = \sqrt{\int\limits_{-\infty}^\infty |t-T(t)|^2d\gamma(t)},
		\end{equation}
		where $\mathbf{W}_2$ stands for the quadratic Wasserstein distance (see \cite{villani2008optimal} for further details on optimal transport). To bound $\mathbf{W}_2$, we shall invoke known results about the central limit theorem.
		
				The central limit theorem in \cite[Theorem 4.1]{rio2009upper} gives
		$$
		\mathbf{W}_2^2 (\gamma, \nu) \leq C L_4,
		$$
		for some absolute constant $C > 0$, and where 
		\begin{align*}
			L_4 &= \sum_i \theta_i^4 \EE |Y_i|^4 \leq   \max_i (\theta_i^2 \cdot \EE[Y_i^4]) 
		\end{align*}			
        Substituting in \eqref{eq:level1T}, we conclude that,
        \begin{align} \label{eq:isoeq1}
        \int \langle \theta, y \rangle S(y) d \tilde \xi(y) \leq I(\vol(S)) + C \max_i (|\theta_i| \cdot \EE[Y_i^4]^{1/2}) 
        \end{align}
    \end{proof}

%    \begin{corollary} \label{corr:isotropiclevelone}
  %   \end{corollary}
\begin{proof}[Proof of Theorem \ref{thm:isotropiclevelone} (2)]
        Fix $w = \int y \cdot S(y) d \tilde \xi(y)$ and let $\theta = w / \| w\|_2$.  Then using \eqref{eq:isoeq1}, 
        \begin{align*}
            \| w \|_2 = \int \langle\theta,y \rangle S(y) d \tilde \xi(y) 
            &  \leq I(\vol(S)) + C \max_i (|\theta_i| \cdot \EE[Y_i^4]^{1/2}) \\
            & = I(\vol(S)) + \frac{C \cdot \max_i(|w_i| \cdot \EE[Y_i^4]^{1/2})}{\|w\|_2}
        \end{align*}
        If $\|w\|_2 \leq I(\vol(S))$ then the claim is obvious.  Without loss of generality, assume $\|w\|_2 \geq I(\vol(S))$, which implies that,
        \begin{align*}
            \| w \|_2 
            & \leq I(\vol(S)) + \frac{C \cdot \max_i(|w_i| \cdot \EE[Y_i^4]^{1/2})}{I(\vol(S))}
        \end{align*}
        \end{proof}

Now we are finally ready to wrap up the proof of Theorem \ref{thm:level1}.  

%We restate the theorem here for the convenience of the reader.
%\begin{theorem*}[Level-1 inequality] 
%		Fix $t>0$ and define
%		\begin{equation} \label{eq:defkappat}
%			\kappa_t := \max_i (\partial_i g(X(t)))^2.    
%		\end{equation}
%		Then,
%		$$
%		\|\sigma_t \nabla g(X(t))\|_2 \leq C_1\frac{\sqrt{\kappa_t}}{I(g(X(t)))} + I(g(X(t))),
%		$$
%		for some absolute constant $C_1 > 0$.
%	\end{theorem*}

    \begin{proof}[Proof of Theorem \ref{thm:level1}]
		Let $\xi_t$ be the law of $X(\infty) | X(t)$ and let $\tilde \xi_t$ be the push forward of $\xi_t$ under the linear map $L(x) = \sigma_t^{-1} (x - X(t))$. By Lemma \ref{lem:momentsprocess} $\tilde{\xi}_t$ is isotropic, it is centered and its covariance matrix is the identity. The martingale property of $X(t)$ also implies that for $i = 1,\dots, n$, $$\PP(X_i(\infty) = 1|X_i(t))= \frac{1 + X_i(t)}{2} \text{ and } \PP(X_i(\infty) = -1|X_i(t))= \frac{1 - X_i(t)}{2}.$$\iffalse It follows that 
	\[
           \PP\left(L(X(\infty))_i = \sqrt{\frac{1-X_i(t)}{1+X_i(t)}}|X(t)\right) =\frac{1+X_i(t)}{2}\text{ and  }\PP\left(L(X(\infty))_i = \sqrt{\frac{1+X_i(t)}{1-X_i(t)}}|X(t)\right) = \frac{1-X_i(t)}{2}.
        \]
        \fi
		Thus, a straightforward calculation shows,
		$$
		\sigma_t \nabla g(X(t)) = \int_{\{-1,1\}^n}L(x)g(x)d\xi_t(x) = \int_{L \left (\{-1,1\}^n\right ) } x S( x ) d\tilde \xi_t (dx),
		$$
		where $S = g \circ L^{-1}$.  We will now apply the level-1 inequality (Theorem \ref{thm:isotropiclevelone}) for isotropic product measures on the distribution $\tilde \xi_t$.  
		To this end, we begin by calculating the fourth moment of the coordinates of $Y \sim \tilde \xi_t$.  
		\begin{align*}
			\EE[Y_i^4] = \EE\left[\frac{(X_i(\infty) - X_i(t))^4}{(1-X_i(t)^2)^2}|X(t)\right] &= \frac{1+X_i(t)}{2}\frac{(1-X_i(t))^4}{(1-X_i(t)^2)^2} + \frac{1-X_i(t)}{2}\frac{(-1-X_i(t))^4}{(1-X_i(t)^2)^2}\\
			&= \frac{(1-X_i(t))^3}{2(1-X_i(t)^2)} + \frac{(1+X_i(t))^3}{2(1-X_i(t)^2)}\\
			&= \frac{1+3X_i(t)^2}{1-X_i(t)^2} \leq 4(\sigma_t)^{-2}_{i,i}.
		\end{align*}

		Now using Theorem \ref{thm:isotropiclevelone} on the distribution $\tilde \xi_t$ and the function $S = g \circ L^{-1}$, we get the desired claim.
		\begin{align*}
		 \| \sigma_t \nabla g(X(t)) \|_2 & \leq I(g(X_t)) + C_1 \cdot \frac{\max_i\left(|(\sigma_t)_{i,i} \partial_i g(X_t)| \cdot ((4\sigma_t)^{-2}_{i,i})^{1/2}\right)}{ I(g(X_t))} \\
		 & \leq I(g(X_t)) + 2C_1 \cdot \frac{\max_i |\partial_i g(X_t)| }{ I(g(X_t))} 
		\end{align*}
		
\end{proof}

	\section{Influences remain small} \label{sec:influences}
	This section aims to show that influences remain small, with high probability, over the paths of $X(t)$. Recall that we are dealing with $g_\eps$, a noisy version of $g$.
	In particular, $g_\eps$ has the following Fourier expansion,
	\begin{equation} \label{eq:fourierdecomp}
	    g_\eps = \sum\limits_{A}e^{-\eps|A|}\hat{g}(A)\chi_A.
	\end{equation}
	
	\begin{proof}[Proof of Lemma \ref{lem:smallinfluences}]
		From the representation in \eqref{eq:fourierdecomp} it is readily seen, 
		\begin{equation*}
			\EE\left[\|\nabla T_{e^{-\eps}} g(X(\infty))\|_2^2\right] \leq \sum\limits |A|^2e^{-2\eps|A|}\hat{g}(A)^2 \leq \frac{1}{\eps^2}\sum\limits_{|A| > 0} \hat{g}(A)^2 = \frac{\mathrm{Var}(g(X(\infty)))}{\eps^2},
		\end{equation*}
		where the inequality follows since $x^2e^{-2x} \leq 1.$
		
		For $i \in [n]$ compute,
		\begin{align*}
			\EE\left[|\partial_i T_{e^{-\eps}}g(X(t))|^{2 + e^{-t}}\right] &\leq \EE\left[|\partial_i  T_{e^{-\eps}}g(X(\infty))|^{2}\right]^{1 + \frac{e^{-t}}{2}}\leq \EE\left[|\partial_i T_{e^{-\eps}}g(X(\infty))|^{2}\right]\EE\left[|\partial_i g(X(\infty))|^{2}\right]^{\frac{e^{-t}}{2}}\\
			&= \EE\left[|\partial_i T_{e^{-\eps}} g(X(\infty))|^{2}\right]\mathrm{Inf}_i(g)^{\frac{e^{-t}}{2}}\leq \EE\left[|\partial_i T_{e^{-\eps}} g(X(\infty))|^{2}\right]\kappa^{\frac{e^{-t}}{2}}.
		\end{align*}
		The first inequality is hypercontractivity, Lemma \ref{lem:hyper}, applied to $T_{e^{-\eps}} g$.
		Hence, by first applying a union bound and then Doob's maximal inequality to the submartingales $|\partial_i g_\eps(X_s)|^{2 + \frac{e^{-t}}{2}}$,
		
		\begin{align*}
			\PP \left (\max_{s \in [0,t], i \in [n]} |\partial_i g_\eps(X_s)| \geq \alpha \right ) &\leq \sum\limits_i\PP \left (\max_{s \in [0,t]} |\partial_i g_\eps(X_s)| \geq \alpha \right )\leq \frac{1}{\alpha^{2 + e^{-t}}}\sum\limits_i \EE\left[|\partial_i g_\eps(X(t))|^{2 + e^{-t}}\right] \\
			&\leq  \frac{\kappa^{\frac{e^{-t}}{2}}}{\alpha^{2 + e^{-t}}}\EE\left[\|\nabla T_{e^{-\eps}} g\|_2^{2}\right] \leq \frac{\kappa^{\frac{e^{-t}}{2}}\mathrm{Var}(g)}{\alpha^{2 + e^{-t}}\eps^2}.
		\end{align*}
	\end{proof}
	\section{Further properties of the noise process} \label{sec:properties}
	This section records several useful properties of the process $X(t)$.
	Recall that the process satisfies the following $SDE$:
	$$
	d X(t) = \sigma_t d B_t, X(0) = 0,
	$$
	where $B_t$ is a standard Brownian motion in $\RR^n$ and $\sigma_t$ is the diagonal matrix with 
	$$
	(\sigma_t)_{i,i} := \sqrt{(1+X_i(t)) (1-X_i(t))} \mathbf{1}_{|X_i(t)| \leq 1}.
	$$ 
	Let us prove an immediate consequence of the definition, which shall also explain the choice of matrix $\sigma_t$.
	\begin{lemma} \label{lem:momentsprocess}
		The process $X(t)$ is  martingale which converges almost surely to a limit $X(\infty) \sim \mu^n$, uniform on $\{-1,1\}^n.$
		Moreover, if we fix $t \geq 0$ and let $Y_t$ have the (random) law of $X(\infty)|X(t)$. Then,
		$$\EE\left[Y_t\right] = X(t), \text{ and  } \mathrm{Cov}(Y_t) = \sigma_t^2.$$
	\end{lemma}
	\begin{proof}
		First note that, for every $i = 1,\dots,n$, by construction $X_i(t) \in [-1,1]$ almost surely. Thus, by the martingale convergence theorem $\lim\limits_{t\to \infty}X_{i}(t)= X_i(\infty)$, for some random variable $X_i(\infty)$. From the definition of $(\sigma_t)_{i,i}$ it is clear than $X_i(\infty)$ is supported on $\{-1,1\}$, and, since $\EE\left[X_i(\infty)\right] = \EE[X_i(0)] = 0$ we conclude that $X_i(\infty)$ is uniform on $\{-1,1\}$. Since $X(t)$ has independent coordinates, the claim about $X(\infty)$ follows.
		
		For the second part, $X(t)$ is a martingale, hence $\EE\left[Y_t\right]=\EE\left[X(\infty)|X(t)\right] = X(t)$, and because $X(\infty)$ is uniform on $\{-1,1\}^n$, we have $\mathrm{Cov}(X(\infty)|X(t))_{i,i} = \EE\left[X_i(\infty)^2|X_i(t)\right] = (1-X_i(t)^2) = \sigma_{i,i}^2$, for every $i = 1,\dots, n$. The off-diagonal elements of $\mathrm{Cov}(X(\infty)|X(t))$ are $0$, since $X(t)$ has independent coordinates. 
	\end{proof}
	The process $X(t)$ also possesses hyper-contractive properties. Below we prove such a result, which is used in the proof of Lemma \ref{lem:smallinfluences}.
	\begin{lemma} \label{lem:hyper}
		For any multi-linear $f:[-1,1]^n \to \RR$ and $t > 0$, if $\delta(t) = e^{-t}$, then
		$$\EE\left[|f(X(t))|^{2+\delta(t)}\right]^{\frac{1}{2+\delta(t)}} \leq \EE\left[|f(X({\infty}))|^{2}\right]^{\frac{1}{2}}.$$
	\end{lemma}
	\begin{proof}
		The proof goes by induction. We first show that the claim in dimension $n$ can be reduced to the same claim in dimension $n-1$. The second part of the proof will be to prove the claim when $n=1$.
		\paragraph{Inductive argument:}
		If $f:[-1,1]^n \to \RR$ is multi-linear, we may write it as $f(x)= x_nf'(\tilde{x}) + f''(\tilde{x})$, where $f',f'':[-1,1]^{n-1}$ are multi-linear and $\tilde{x} = (x_1,\dots x_{n-1})$. Thus, writing $\tilde{X}(t) =(X_1(t),\dots X_{n-1}(t))$,
		\begin{align*}
			\EE\left[|f(X(t))|^{2+\delta(t)}\right]^{\frac{1}{2+\delta(t)}} &= \EE\left[|X_n(t)f'(\tilde{X}(t)) + f''(\tilde{X}(t))|^{2+\delta(t)}\right]^{\frac{1}{2+\delta(t)}} \\
			&=\EE_{\tilde{X}(t)}\left[\EE_{X_n(t)}\left[|X_n(t)f'(\tilde{X}(t)) + f''(\tilde{X}(t))|^{2+\delta(t)}\right]\right]^{\frac{1}{2+\delta(t)}}\\
			&\leq \EE_{\tilde{X}(t)}\left[\EE_{X_n(\infty)}\left[|X_n(\infty)f'(\tilde{X}(t)) + f''(\tilde{X}(t))|^{2}\right]^{\frac{2+\delta(t)}{2}}\right]^{\frac{1}{2+\delta(t)}}\\
			&\leq \EE_{X_n(\infty)}\left[\EE_{\tilde{X}(t)}\left[|X_n(\infty)f'(\tilde{X}(t)) + f''(\tilde{X}(t))|^{2 +\delta(t)}\right]^{\frac{2}{2+\delta(t)}}\right]^{\frac{1}{2}}\\
			&\leq \EE_{X_n(\infty)}\left[\EE_{\tilde{X}_\infty}\left[|X_n(\infty)f'(\tilde{X}(\infty)) + f''(\tilde{X}(\infty))|^{2 }\right]\right]^{\frac{1}{2}}\\
			&= \EE\left[|f(X(\infty))|^2\right]^\frac{1}{2}.
		\end{align*}
		The first inequality uses the induction hypothesis on the uni-variate affine function, $X_n(t)f'(\tilde{X}(t)) + f''(\tilde{X}(t))$, when $\tilde{X}(t)$ is fixed. Similarly, the third inequality is the induction hypothesis applied to the multi-linear function $X_n(\infty)f'(\tilde{X}(t)) + f''(\tilde{X}(t))$, when $X_n(\infty)$ is fixed. The second inequality is Minkowski's integral inequality.
		\paragraph{A one dimensional inequality:}
		We now prove the base case of the induction when $n=1$. Let $f:[-1,1] \to \RR$ be multi-linear. We first assume that $f(x) \geq 0$ for every $x \in [-1,1]$. Hence $f(x) = ax+b$, for some $a,b \in \RR$ with $|a| \leq b$. By re-scaling, we may assume $b=1$ and $|a| < 1$.
		Thus, we are interested in bounding $\EE\left[|aX(t)+1|^{2+\delta(t)}\right]$. Taking a third order Taylor approximation  of $x \to (1+x)^{2+\delta(t)}$ and using the bound
		$$\frac{d^4}{dx^4}(1+x)^{2+\delta(t)} \leq 0,$$
		when $x > -1$ and $\delta(t) \leq 1$, we get,
		\begin{align*}
			\EE&\left[|aX(t)+1|^{2+\delta(t)}\right] \\
			&\leq 1 + a(2+\delta(t))\EE[X(t)] + a^2\frac{(2+\delta(t))(1+\delta(t))}{2}\EE[X(t)^2]\\
			&\ \ \  + a^3\frac{(2+\delta(t))(1+\delta(t))\delta(t)}{6}\EE[X(t)^3]\\
			&= 1 + a^2\frac{(2+\delta(t))(1+\delta(t))}{2}(1-e^{-t}).
		\end{align*}
		The last identity uses the fact that $X(t)$ is symmetric as well as Lemma \ref{lem:noisetonoise}.
		Since $\delta(t) = e^{-t}$, we get,
		\begin{align*}
			\EE\left[|aX(t)+1|^{2+\delta(t)}\right]^{\frac{2}{2+\delta(t)}} &\leq (1 + a^2\frac{(2+\delta(t))(1+\delta(t))}{2}(1-e^{-t}))^{\frac{2}{2+\delta(t)}}\\
			&\leq 1 + a^2(1+\delta(t))(1-e^{-t}) = 1+(1 -e^{-2t})a^2\\
			&\leq 1+a^2 = \EE\left[(1+aX(\infty))^2\right].
		\end{align*}
		Which finishes the proof when $f$ is non-negative. For general $f$, consider the non-negative affine function $g:[-1,1]\to \RR$ defined by $g(-1) = |f(-1)|$ and $g(1) = |f(1)|$. By Jensen's inequality, we have $|f(x)| \leq |g(x)|$, for every $x \in [-1,1]$, indeed, since $f$ and $g$ are affine
		$$|f(x)| = \left|\frac{f(-1)(1-x) + f(1)(1+x)}{2} \right| \leq \frac{|f(-1)|(1-x) + |f(1)|(1+x)}{2} =|g(x)|.$$
		Thus,
		$$\EE\left[|f(X(t))|^{2+\delta(t)}\right]^{\frac{1}{2+\delta(t)}} \leq \EE\left[|g(X(t))|^{2+\delta(t)}\right]^{\frac{1}{2+\delta(t)}}\leq \EE\left[|g(X(\infty))|^{2}\right]^{\frac{1}{2}} = \EE\left[|f(X(\infty))|^{2}\right]^{\frac{1}{2}}.$$
		The second inequality holds since $g$ is non-negative and the last inequality follows from the fact that $X(\infty) \in \{-1,1\}$ almost surely. 
	\end{proof}

  \newpage
  
	\bibliographystyle{plain}
	\bibliography{bib}{}

\end{document}